\documentclass[11pt]{amsart}
\usepackage[centertags]{amsmath}
\usepackage{amsfonts,amssymb}
\usepackage{graphicx}
\usepackage{enumerate}
 \usepackage{url}
\usepackage{caption}
  \newtheorem{theorem}{Theorem}
  \newtheorem{corollary}{Corollary}

  \newtheorem{lemma}{Lemma}

\begin{document}

\title{Explicit bounds for Dickman's function}

\author{Andreas Weingartner} 
%\contrib[...]{...} 
\address{Department of Mathematics, Southern Utah University, 351 West University Boulevard, Cedar City, Utah 84720, USA} 
\email{weingartner@suu.edu} 
%\urladdr{...} 
%\dedicatory{...} 
%\date{\today} 
%\thanks{...}
%\translator{...} 
%\keywords{...}
%\subjclass{11N25} 

\begin{abstract}
Dickman's function $\rho(u)$ denotes the natural density of integers $n$ whose largest prime factor does not exceed $n^{1/u}$. 
We establish numerically explicit upper and lower bounds for $\rho(u)$, 
resulting in estimates with a relative error of less than $0.005/u^2$ for all $u\ge 5$.
This allows for an approximate evaluation of $\rho(u)$ without the need to solve the delay differential equation numerically.
\end{abstract}

\maketitle

\section{Introduction}
Dickman's function $\rho(u)$ is defined as the continuous solution to the delay differential equation
$u \rho'(u)+\rho(u-1)=0$ with initial condition $\rho(u)=1$ for $0\le u\le 1$. 
It represents the natural density of integers $n$ whose largest prime factor does not exceed $n^{1/u}$ \cite{Dick}.

For $u>1$, let $\xi=\xi(u)$ be the unique positive solution of $e^\xi=1+u\xi$ and put $\xi(1)=0$. 
Define 
$$
I(s):= \int_0^s \frac{e^z-1}{z}dz.
$$

De Bruijn \cite{BruRho} showed that $\rho(u)\sim \tilde{\rho}(u)$ as $u\to \infty$, where
$$
\tilde{\rho}(u) := \frac{\exp\{\gamma - u \xi + I(\xi)\}}{\sqrt{2 \pi I''(\xi)}}
$$
and $\gamma$ is Euler's constant.
Alladi \cite{All} sharpened this to $\rho(u)= \tilde{\rho}(u)(1+O(1/u))$.
Our goal is to make the estimate entirely explicit,
providing a convenient way to obtain approximate values of $\rho(u)$
without the need to solve the delay differential equation numerically.

\begin{theorem}\label{thm1}
For $u>1$ we have
$$
\tilde{\rho}(u) \left(1-\frac{1}{12u}\right)<\rho(u)< \tilde{\rho}(u) \left(1-\frac{1}{12u(1+\frac{1}{\log u})}\right) .
$$
\end{theorem}

Table \ref{table1} lists numerical examples based on Theorems \ref{thm1} and \ref{thm2}, the first four of which are 
consistent with the values obtained by van de Lune and Wattel \cite{LunWat}, who iterate the trapezoidal rule to 
approximate $\rho(u)$ to at least five significant digits for many values of $u\le 1000$.
\begin{table}[h] 
  \centering 
  \begin{tabular}{|r|l|l|r|} 
    \hline
    $u$ & $a$: Thm. \ref{thm1} & $a$: Thm. \ref{thm2} & $b$ \\ 
    \hline \hline
    5 & 3.5... & 3.54... & 4 \\ 
    \hline
%10 & 2.7... & 2.770... & 11 \\ 
%   \hline
    20 & 2.46... & 2.461... & 29 \\
    \hline
   100 & 1.000... & 1.000595... & 229 \\
\hline
   1000 & 4.5876... & 4.5876682... & 3464 \\
   \hline
   10000 & 4.8487... & 4.848709620... & 45867 \\
\hline
  \end{tabular}
   \caption{Values of $\rho(u) = a \cdot 10^{-b}$ from Theorems \ref{thm1} and \ref{thm2}.} % Place here for top alignment
   \label{table1} 
\end{table}

We will derive Theorem \ref{thm1} from Theorem \ref{thm2}, which involves the expression
$$
\alpha(u):=\frac{1}{8} \frac{I^{(4)}(\xi)}{(I^{(2)}(\xi))^2} - \frac{5}{24} \frac{(I^{(3)}(\xi))^2}{(I^{(2)}(\xi))^3} \sim -\frac{1}{12u},
$$
where the derivatives $I^{(k)}(s)$ can easily be found from $I'(s)=(e^s -1)/s$.

\begin{theorem}\label{thm2}
We have
$$
\rho(u) =  \tilde{\rho}(u) \left(1+ \alpha(u) + \frac{\theta(u)}{ u^2}\right),
$$
where $|\theta(u)| < 0.005$ for $u\ge 5$, and $0<\theta(u)<0.005$ for $u\ge 9$. 
\end{theorem}

Figure \ref{fig1} shows the graph of $\rho(u)/\tilde{\rho}(u)$, the bounds from Theorem \ref{thm1} and
the estimate $1+\alpha(u)$ from Theorem \ref{thm2}, for $1< u \le 10$.

\begin{figure}[htbp]
    \centering
    \includegraphics[width=0.92\textwidth]{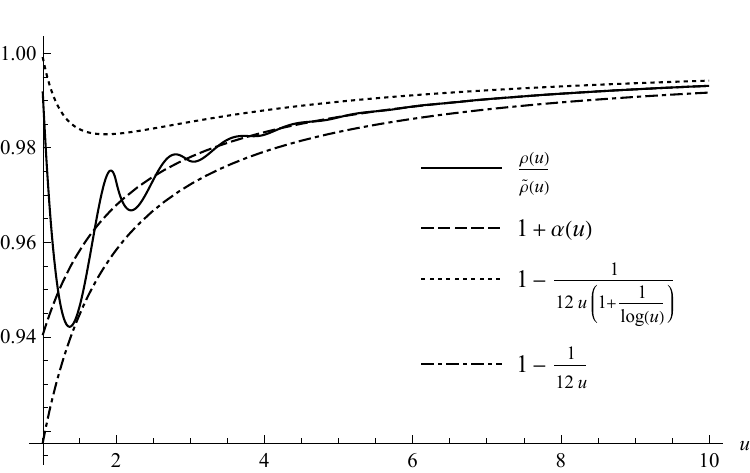}
    \caption{Theorems \ref{thm1} and \ref{thm2} for $1< u \le 10$.}
    \label{fig1}
\end{figure}

Dickman \cite{Dick} showed that  $\rho(u)=\lim_{x\to \infty} \Psi(x,x^{1/u})/x$, where
$\Psi(x,y)$ denotes the number of positive integers up to $x$ whose largest prime factor does not exceed $y$.
Pomerance \cite[(1.25)]{Gran} asked if $\Psi(x,y)\ge x\rho(u)$ holds for $x/2\ge y\ge 1$, where $u=(\log x)/\log y$.
Gorodetsky \cite{Gor} confirmed this for $x$ sufficiently large, assuming the Riemann hypothesis and 
one additional hypothesis about the error term in the prime number theorem \cite[(1.7)]{Gor}.
Together with Theorem \ref{thm1} or Theorem \ref{thm2}, we obtain conditional lower bounds for
$\Psi(x,y)$ that are easy to evaluate numerically.

Hildebrand and Tenenbaum \cite{HT}, Smida \cite{Smida} and Xuan \cite{Xuan} derive asymptotic expansions for $\rho(u)$, but without being numerically explicit.

To compute $\tilde{\rho}(u)$ (or $\log \tilde{\rho}(u)$ when $u$ is large) and $\alpha(u)$, note that
$$
I(\xi)=\mathrm{Ei}(\xi)-\gamma - \log \xi,
$$
where $\mathrm{Ei}(z)$ is the exponential integral (\texttt{ExpIntegralEi[z]} in Mathematica, \texttt{real(-incgam(0,-z))} in PARI/GP for $z\in \mathbb{R}^+$), 
and
\begin{equation*}\label{eqximat}
\xi = \xi(u) =-W_{-1}\left(-\frac{e^{-1/u}}{u}\right)-\frac{1}{u},
\end{equation*}
where $W_{-1}(z)$ is the lower of the two real branches, for $-1/e\le z<0$,
of the Lambert W function (\texttt{ProductLog[-1,z]} in Mathematica, \texttt{lambertw(z,-1)} in PARI/GP).

Buchstab's function $\omega(u)$ is the unique continuous solution, for $u\ge 1$, of the delay differential equation
$(u\omega(u))'=\omega(u-1)$, with initial condition $\omega(u)=1/u$ for $1\le u \le 2$. 
It arises in connection with the distribution of integers free of small prime factors \cite[Ch.~III.6]{Ten}.
Corollary \ref{cor1} shows that $\tilde{\rho}(u)$ 
can serve as a convenient and numerically explicit upper bound for $|\omega(u)-e^{-\gamma}|$,
although it is not of the smallest possible order of magnitude \cite{Hil}. 
We obtain more accurate and numerically explicit estimates in \cite{WeinB}.

\begin{corollary}\label{cor1}
For $u\ge 1 $, we have 
$$
|\omega(u)-e^{-\gamma}| < \rho(u) < \tilde{\rho}(u),
\qquad
|\omega'(u)| \le  \rho(u) < \tilde{\rho}(u).
$$
\end{corollary}

\begin{proof}
The inequality $\rho(u)<\tilde{\rho}(u)$ follows from Theorem \ref{thm1}. 
Tenenbaum \cite[Thm.~III.6.6]{Ten} shows that $|\omega'(u)|\le \rho(u)$,
which implies 
$$|\omega(u)-e^{-\gamma}| = |\int_u^\infty \omega'(t) dt |\le \int_u^\infty \rho(t) dt < \int_u^\infty -\rho'(t) dt =\rho(u),$$
where the last inequality follows from
$$
\rho(t) = \frac{1}{t} \int_{t-1}^t \rho(v) dv < \frac{1}{t} \rho(t-1) = -\rho'(t) \qquad (t>1).
$$
\end{proof}

\section{Proof of Theorem \ref{thm2}}

We follow the proof of Tenenbaum \cite[Thm.~III.5.13]{Ten}, 
where Alladi's result \cite[(3.9)]{All} with a relative error of $O(u^{-1})$ is established
via the saddle point method. To obtain a relative error of $<0.005 u^{-2}$, 
we keep track of all constant factors, divide the range of integration into one extra interval,
 and use a Taylor polynomial of order seven instead of three near the saddle point. 
For small values of $u$, we verify the result by calculating $\rho(u)$ numerically as described in Section \ref{SecNum}.

We will need the following lemmas.

\begin{lemma}\label{lemIk}
For real $x>0$ and natural numbers $k$ we have 
$$
\frac{e^x -1}{x} \ge I^{(k)}(x) \ge I^{(k+1)}(x) \quad (k\ge 1),
$$
$$
I^{(k)}(x) \ge \frac{e^x}{x} \left(1-\frac{k-1}{x}\right) \qquad (k\ge 2),
$$
$$
I^{(k)}(x) \le  \frac{e^x}{x} \left(1-\frac{k-1}{x}+\frac{(k-1)(k-2)}{x^2}\right) \qquad (k\ge 3).
$$
\end{lemma}
\begin{proof}
For $k\ge 1$, we have $I^{(k)}(x) = \int_0^1 h^{k-1} e^{hx} dh$, which implies the first inequality.
The other two follow from applying integration by parts to this integral.
\end{proof}

\begin{lemma}\label{lemxi}
For $u\ge 2$ we have 
$$
\log (u\log u) < \xi(u) < \log (u\log u) +1.
$$
\end{lemma}
\begin{proof}
This is an easy exercise, following the same reasoning as in the proof of \cite[Lem.~III.5.11]{Ten}.
\end{proof}

\begin{lemma}\label{lemtail}
For real $\delta >0$, $A>0$, and $1\le m < \delta^2 A+1$, we have
$$
\int_\delta^\infty t^m e^{-A t^2/2} dt \le \frac{\delta^{m+1}}{\delta^2 A - (m-1)} e^{-A\delta^2/2}.
$$
\end{lemma}
\begin{proof}
This follows from a change of variables in the bound \cite[Prop.~2.7]{Pin} 
$$
\Gamma(a,x):=\int_x^\infty t^{a-1} e^{-t} dt \le \frac{x^a e^{-x}}{x-(a-1)}
$$
for real $x$, $a$ with $x>a-1\ge 0$.
\end{proof}

\begin{proof}[Proof of Theorem \ref{thm2}]

As in \cite{Ten}, we evaluate the inverse Laplace integral
\begin{equation}\label{eqlapinv}
\begin{split}
\rho(u) & =\frac{1}{2\pi i} \int_{-\xi-i\infty}^{-\xi+i\infty} \widehat{\rho}(s) e^{us} ds
=\frac{1}{2\pi i} \int_{-\xi-i\infty}^{-\xi+i\infty}  e^{\gamma + I(-s)+ us} ds \\
& =\frac{e^{\gamma-u\xi + I(\xi)}}{2\pi} \int_{-\infty}^{\infty} e^{I(\xi-i\tau)-I(\xi)+ui\tau} d\tau,
\end{split}
\end{equation}
where $s=-\xi+i\tau$. 

Let $\delta:= \sqrt{8\log(u)/u}$
and define
$$
H(\tau):=I(\xi)-\mathrm{Re}(I(\xi-i\tau)) =\int_0^1 e^{h\xi} \frac{1-\cos(h\tau)}{h} dh.
$$

When $\delta\le |\tau| \le 1$, we have $1-\cos(h\tau)\ge 0.45 \tau^2 h^2$ for $0\le h\le 1$, so
$$H(\tau) \ge 0.45\tau^2 \int_0^1 h e^{h\xi} dh = 0.45 I''(\xi) \tau^2$$
and the contribution to \eqref{eqlapinv} is
$$
< \frac{1}{\pi} e^{\gamma-u\xi + I(\xi)} \int_\delta^1 \exp(-0.45 I''(\xi) \tau^2)d\tau 
< \frac{1}{\pi} e^{\gamma-u\xi + I(\xi)} \frac{\exp(-0.45 I''(\xi)\delta^2)}{2\cdot 0.45 I''(\xi)\delta},
$$
hence
$$
<  \tilde{\rho}(u)  \frac{\exp(-0.45 I''(\xi)\delta^2)}{0.45 \delta \sqrt{2 \pi I''(\xi)}} =: \tilde{\rho}(u) E_1(u).
$$

Similarly, when $1 \le |\tau| \le \pi$, we have $1-\cos(h\tau) \ge 2\tau^2 h^2/\pi^2$, which leads to a contribution of
$$
<  \tilde{\rho}(u)  \frac{\exp(-2I''(\xi)/\pi^2)}{(2/\pi^2) \sqrt{2 \pi I''(\xi)}} =: \tilde{\rho}(u) E_2(u).
$$

For the contribution from $\pi \le |\tau| \le 1+u\xi$, we find as in \cite{Ten} that
$$
H(\tau)\ge \frac{u \pi^2}{2(\pi^2+\xi^2)}-\frac{2}{\sqrt{\pi^2+\xi^2}}=: \eta(u),
$$
so that these segments contribute
$$
<\frac{1}{\pi} e^{\gamma-u\xi + I(\xi)} (1+u\xi-\pi) e^{-\eta(u)},
$$
hence
$$
< \tilde{\rho}(u) \sqrt{2I''(\xi)/\pi} (1+u\xi-\pi) e^{-\eta(u)}=: \tilde{\rho}(u) E_3(u).
$$

When $|\tau|\ge 1 +u\xi$ and $s=-\xi+i\tau$, we have $\widehat{\rho}(s)=e^{-J(s)}/s$ by \cite[(5.51)]{Ten}, where 
$$
|J(s)| := \left|\int_0^\infty \frac{e^{-s-t}}{s+t} dt \right|\le \frac{e^{\xi}}{|\tau|}=\frac{1+u\xi}{|\tau|} \le 1.
$$ 
Thus
$$
|e^{-J(s)}-1| \le |J(s)| (e-1) \le (e-1)\frac{1+u\xi}{|\tau|}
$$
and
$$
\widehat{\rho}(s) = \frac{1}{s}+ \frac{\nu_{\tau,u}}{s}(e-1) \frac{1+u\xi}{|\tau|},
$$
where $|\nu_{\tau,u}|\le 1$.
Substituting this into \eqref{eqlapinv} and applying integration by parts to the first term, 
we find that the contribution from $|\tau|\ge 1 +u\xi$ to \eqref{eqlapinv} is 
$$
\le \frac{e^{-u\xi}}{\pi}\left(\frac{2}{u(1+u\xi)}+(e-1)\right) < \frac{e^{-u\xi}}{\pi}(e+1),
$$
hence
$$
< \tilde{\rho}(u) (e+1)e^{-\gamma-I(\xi)} \sqrt{2I''(\xi)/\pi} =: \tilde{\rho}(u) E_4(u).
$$

For $-\delta \le \tau \le \delta$, we use the Taylor expansion of $I(\xi-i\tau)$ around $\tau=0$. Since
$$
|I^{(k)}(\xi-i\tau)|=|\int_0^1 h^{k-1} e^{h(\xi-i\tau)} dh | \le I^{(k)}(\xi) \le I'(\xi) = u,
$$
for $k\ge 1$, we have
$$
I(\xi-i\tau)-I(\xi)+ui\tau = \sum_{k=2}^7 \frac{I^{(k)}(\xi)}{k!}(-i\tau)^k + \frac{I^{(8)} (\xi)\ \mu_{u,\tau}}{8!} \tau^8,
$$
where $|\mu_{u,\tau}|\le 1$. 
Thus,
\begin{equation}\label{eqhint}
\int_{-\delta}^{\delta} e^{I(\xi-i\tau)-I(\xi)+ui\tau }d\tau = \int_{-\delta}^\delta e^{-I''(\xi)\tau^2/2}h(u,\tau)d\tau, 
\end{equation}
where 
$$
h(u,\tau):= \exp\left\{ \sum_{k=3}^7 \frac{I^{(k)}(\xi)}{k!}(-i\tau)^k + \frac{I^{(8)} (\xi)\mu_{u,\tau}}{8!} \tau^8\right\}.
$$
Writing $M$ for the exponent, we have
$$
h(u,\tau)=e^M=\sum_{j=0}^5 \frac{M^j}{j!}+\frac{1.1 M^6 \nu_M}{6!} \quad (\mathrm{Re}( M)  \le \log 1.1),
$$
where $|\nu_M|\le 1$. 
Since $|\tau|\le \delta $ and $I^{(k+1)}(\xi)\le I^{(k)}(\xi)$, we have
\begin{equation*}\label{MUB}
\mathrm{Re}( M) \le  \frac{I^{(4)} (\xi)}{4!}\tau^4 +\frac{I^{(8)} (\xi)}{8!}\tau^8  
 \le \frac{I^{(4)} (\xi) \delta^4}{4!}\left(1+\frac{\delta^4}{1680}\right),
\end{equation*}
which implies that $\mathrm{Re}( M)<\log 1.1$ for $u\ge 980$.

Let $h_k$ be $M^k$ without all the terms that contribute zero to the integral in \eqref{eqhint} and write
$$
I_k:=I^{(k)}(\xi), \qquad a_k:= \frac{I^{(k)}(\xi)}{k!}=\frac{I_k}{k!}.
$$
We have
$$h_0=1,$$ 
$$
h_1=a_4 \tau^4 - a_6  \tau^6 +   \theta_1 \tau^8 R_1(\tau),
$$
$$
h_2= -a_3^2 \tau^6 +2a_3a_5\tau^8 +a_4^2 \tau^8
- \tau^{10} S_2
+\theta_2  \tau^{10} R_2(\tau),
$$
$$
h_3=-3a_3^2 a_4\tau^{10} +     \tau^{12} S_3 +  \theta_3   \tau^{12} R_3(\tau)
$$
$$
h_4=a_3^4 \tau^{12} - \tau^{14} S_4 + \theta_4  \tau^{14} R_4(\tau),
$$
$$
h_5 =\tau^{16} S_5 + \theta_5 \tau^{16} R_5(\tau),
$$
where $|\theta_i|\le 1$ and $S_1=0$, $R_1(\tau):=I_8/8!$, 
$$
S_2:= \sum_{3\le i,j\le 8\atop \ i+j = 10 }\frac{I_i I_j}{i!j!},\quad
R_2(\tau):= \sum_{3\le i,j\le 8, \ i+j \ge 11 \atop  i+j \equiv 0 \bmod 2 \text{ or } \max(i,j)=8}\frac{\tau^{i+j-10}}{i!j!}I_i I_j,
$$
$$
S_3:=\sum_{3\le i,j,k\le 8, \atop  i+j+k = 12  }\frac{I_i I_j I_k}{i!j!k!} , \quad
R_3(\tau):=\sum_{3\le i,j,k\le 8, \ i+j+k \ge 13 \atop  i+j +k\equiv 0 \bmod 2 \text{ or } \max(i,j,k)=8}\frac{\tau^{i+j+k-12}}{i!j!k!} I_i I_j I_k,
$$
et cetera.

With the help of Lemma \ref{lemIk}, we verify that $I_8 < I_2^4 u^{-3}$ for $u\ge 14=:u_1$; 
each product $I_i I_j$ appearing in $S_2$ or $R_2$ satisfies $I_i I_j < I_2^5  u^{-3}$ for $u\ge 5=:u_2$;
each product $I_i I_j I_k$ appearing in $S_3$ or $R_3$ satisfies $I_i I_j I_k < I_2^6  u^{-3}$ for $u\ge 3=:u_3$, et cetera,
with $u_4=u_5=1$. 
Since $|\tau|\le \delta \le 1/4$, which holds for $u\ge 870$, we have $|R_j(\tau)| \le R_j(1/4) $.
Let $ \alpha_j, \beta_j$ denote the values of the
sums $S_j$ and $R_j(1/4)$, respectively, but without the products of the $I$'s. Then
$$
0\le  S_j \le \alpha_j I_2^{j+3} u^{-3},\quad |R_j(\tau)| \le \beta_j  I_2^{j+3} u^{-3}.
$$
To exploit the alternating $\pm$ pattern in front of the $S_j$'s, we define $\gamma_1^+=\gamma_1^-=\beta_1=1/8!$; 
for $1\le j \le 2$, let  
$\gamma^+_{2j}:=\beta_{2j}$, 
$\gamma^-_{2j}:=\alpha_{2j}+\beta_{2j}$,
$\gamma^+_{2j+1}:=\alpha_{2j+1}+\beta_{2j+1}$,
$\gamma^-_{2j+1}:=\beta_{2j+1}$.

Since $I_{k+1}\le I_k$, we have
$$
|M|\le \sum_{k=3}^8 \frac{I_k}{k!}|\tau|^k \le \frac{I_3 |\tau|^3}{3!} 1.065761 \quad  (|\tau|\le \delta \le 1/4),
$$
and therefore
$$
|1.1 M^6| \le \frac{I_3^6 \tau^{18}}{(3!)^6} 1.612 <  \frac{u^{-3} I_2^9 \tau^{18}}{(3!)^6} 1.612
 =: u^{-3} I_2^9 \tau^{18} \gamma_6,
$$
as $I_3^6< u^{-3}I_2^9$ for $u\ge 1$. 
With $\gamma_6^+=\gamma_6^-=\gamma_6$, 
the contribution of the $S_j$'s, the $R_j$'s and $|1.1 M^6|$ to \eqref{eqhint} is 
\begin{multline*}
< \sum_{k=1}^6 \frac{\gamma_k^+  I_2^{3+k}}{u^3 k!} \int_{-\infty}^\infty e^{-I''(\xi)\tau^2/2} \tau^{6+2k} d\tau
= \sqrt{\frac{2\pi}{I''(\xi)}}\sum_{k=1}^6 \frac{\gamma_k^+  I_2^{3+k}}{u^3 k!} \frac{(5+2k)!!}{I_2^{3+k}}\\
< \sqrt{\frac{2\pi}{I''(\xi)}}\frac{1}{u^3}\sum_{k=1}^6 \frac{\gamma_k^+ (5+2k)!!}{k!}
< \sqrt{\frac{2\pi}{I''(\xi)}} \frac{5.39}{u^3},
\end{multline*}
and, similarly, it is
$$
>- \sqrt{\frac{2\pi}{I''(\xi)}}\frac{1}{u^3}\sum_{k=1}^6 \frac{\gamma_k^- (5+2k)!!}{k!}
>- \sqrt{\frac{2\pi}{I''(\xi)}} \frac{4.37}{u^3},
$$
where $n!!$ is the product of all the positive integers up to $n$ that have the same parity as $n$.

Writing the main terms as 
$g(\tau):=1+\sum_{j=2}^6 c_j(u) \tau^{2j}$, Lemma \ref{lemtail} yields
$$
\int_\delta^\infty  e^{-I''(\xi)\tau^2/2} |g(\tau)| d\tau 
\le \frac{e^{-I''(\xi) \delta^2/2}}{I''(\xi) \delta} \left(1 + \frac{3}{2}\sum_{j=2}^6 |c_j(u)| \delta^{2j} \right) , 
$$
provided $I''(\xi)\delta^2 \ge 33$, which is the case if $u\ge 127$. 
We find that
$$
1 + \frac{3}{2}\sum_{j=2}^6 |c_j(u)| \delta^{2j} \le 7
$$
for $u\ge 380$, which implies
$$
2\int_\delta^\infty  e^{-I''(\xi)\tau^2/2}|g(\tau)| d\tau 
< \frac{ 14e^{-I''(\xi) \delta^2/2}}{ \delta \sqrt{2\pi I''(\xi)} }  \sqrt{\frac{2\pi}{ I''(\xi)}}
=: E_5(u)  \sqrt{\frac{2\pi}{ I''(\xi)}}.
$$

Now $f(u):=u^3\sum_{k=1}^5 E_k(u)=u^{3-(0.45)8+o(1)}=u^{-0.6+o(1)}$ has a limit of zero and is decreasing for $u\ge 70$, while 
$f(922)<0.05$. 

Combining everything, we obtain 
\begin{equation}\label{eqfinrho}
\rho(u) = \tilde{\rho}(u) \left(1+\alpha(u) +\beta(u)+\lambda(u) u^{-3}\right),
\end{equation}
where $|\lambda(u)|<5.39 + 0.05 =5.44$ for $u\ge 980$,
$$
\alpha(u) := \frac{a_4 3!!}{I''(\xi)^2} - \frac{a_3^2 5!!}{2! I''(\xi)^3}=\frac{I^{(4)}(\xi)}{8 I''(\xi)^2}-\frac{5 I^{(3)}(\xi)^2}{24 I''(\xi)^3}
\sim -\frac{1}{12u},
$$
\begin{multline*}
\beta(u):=-\frac{a_6 5!!}{I''(\xi)^3}+\frac{(2 a_3 a_5 + a_4^2) 7!!}{2! I''(\xi)^4}- \frac{3a_3^2 a_4 9!!}{3! I''(\xi)^5}
+\frac{a_3^4 11!!}{4! I''(\xi)^6}\\
=
-\frac{I^{(6)}(\xi)}{48 I''(\xi)^3}
+\frac{7 I^{(3)}(\xi) I^{(5)}(\xi)}{48  I''(\xi)^4}
+\frac{35 I^{(4)}(\xi)^2}{384 I''(\xi)^4}
-\frac{35 I^{(3)}(\xi)^2 I^{(4)}(\xi)}{64 I''(\xi)^5}
+\frac{385 I^{(3)}(\xi)^4}{1152 I''(\xi)^6}\\
 \sim \frac{1}{288 u^2}.
\end{multline*}

Lemma \ref{lemIk} shows that $u^2 \beta(u)=  \frac{1}{288}(1-2/\xi +O(1/\xi^2))$,
so that $\lim_{u\to \infty} u^2 \beta(u) = \frac{1}{288}$. Moreover, $u^2 \beta(u)$ is increasing and positive for $u\ge 7$,
so $|u^2 \beta(u) |< \frac{1}{288}$ for $u\ge 7$. 
From \eqref{eqfinrho} we obtain
$$|\theta(u)|:= |u^2(\beta(u) + \lambda(u) u^{-3}) | <\frac{1}{288} + \frac{5.5}{u}.$$
Thus, $|\theta(u)|<0.005$ for $u\ge 3600$. 
For $5\le u \le 3600$, we calculate $\rho(u)$ (resp. $\rho(u)/\tilde{\rho}(u)$) numerically as described in Section \ref{SecNum} and graph
$\theta(u):= (\rho(u)/\tilde{\rho}(u) -1-\alpha(u)) u^2$, to find that $|\theta(u)|\le 0.0036$ in this range.

We have $\theta(u)=u^2\beta(u) +\lambda(u)/u >u^2\beta(u)-5.5/u$ for $u\ge 980$, which implies $\theta(u)>0$ for $u\ge 2250$. 
For $9\le u\le 2250$, we verify that $\theta(u)>0$ by computing $\rho(u)$ as described in Section \ref{SecNum} and graphing $\theta(u)$. 
\end{proof}

In proving Theorem \ref{thm2}, we also proved Theorem \ref{thm3}, where the range $1\le u\le 980$ is again 
established by calculating $\rho(u)$ as in Section \ref{SecNum}. The upper bound $5.44$ is likely far from best possible.
\begin{theorem}\label{thm3}
Equation \eqref{eqfinrho} holds with $|\lambda(u)|<5.44$ for $u\ge 1$. 
\end{theorem}

\section{Proof of Theorem \ref{thm1}}

\begin{lemma}\label{lemAx}
For $u\ge 2$ we have 
$$
 -\frac{\xi}{12 e^\xi}\left(1-\frac{1}{\xi}+\frac{1.75}{\xi^2}\right)< \alpha(u)< -\frac{\xi}{12 e^\xi}\left(1-\frac{1}{\xi}+\frac{0.5}{\xi^2}\right) .
$$
\begin{proof}
Let
$$
A(x):=\frac{1}{8} \frac{I^{(4)}(x)}{(I^{(2)}(x))^2} - \frac{5}{24} \frac{(I^{(3)}(x))^2}{(I^{(2)}(x))^3} ,
$$
so that $\alpha(u)=A(\xi(u))$, and
define $\kappa(x)$ by 
$$
A(x) = -\frac{x}{12 e^x}\left(1-\frac{1}{x}+\frac{\kappa(x)}{x^2}\right).
$$
Lemma \ref{lemIk} shows that $\kappa(x)$ is bounded.
A calculus exercise, aided by a computer, reveals that 
$\kappa(x)$ is increasing on the interval $(0, 2.749...)$, where it reaches its maximum of $1.746...$ at $x=2.749...$, and decreasing 
on $(2.749..., \infty)$ with $\lim_{x\to \infty} \kappa(x) = 1/2$. We have $\kappa(1/3)=0.511...>0.5$. 
The result now follows since $\xi >1/3$ for $u\ge 2$. 
\end{proof}
\end{lemma}

\begin{lemma}\label{lemalpha}
For $u\ge 1$ we have 
$$
-\frac{1}{12 u}
<\alpha(u)-\frac{0.02}{u^2} < \alpha(u)+\frac{0.02}{u^2} 
< - \frac{1}{12 u (1+1/\log u)}.
$$
\end{lemma}
\begin{proof}
For $u\ge 4$, we have $u^2(1-1.75/\xi) >(0.02)  12(u\xi +1)=(0.02) 12  e^\xi$, by Lemma \ref{lemxi}. Lemma \ref{lemAx} 
implies
$$
\alpha(u) -\frac{0.02}{u^2}>  -\frac{\xi}{12 e^\xi}\left(1-\frac{1}{\xi}+\frac{1.75}{\xi^2}\right)-\frac{0.02}{u^2}
>  \frac{-\xi}{12 e^\xi}=\frac{-\xi}{12 (u\xi+1)}> \frac{-1}{12 u}.
$$
For $u\ge 1$, we have $u^2\ge (0.02) 24 \xi e^\xi= (0.02) 24 \xi(u\xi+1)$. Lemma \ref{lemAx} yields
$$
\alpha(u) +\frac{0.02}{u^2} < -\frac{\xi}{12 e^\xi}\left(1-\frac{1}{\xi}+\frac{0.5}{\xi^2}\right)+\frac{0.02}{u^2}
< \frac{-\xi}{12 e^\xi}\left(1-\frac{1}{\xi}\right)=\frac{-(\xi -1)}{12(u\xi+1)}.
$$
We claim that 
$$
-\frac{\xi -1}{12(u\xi+1)}< -\frac{1}{12u(1+1/\log u)}.
$$
Indeed, this is equivalent to $\xi>1+(1+1/u)\log u$, which follows from Lemma \ref{lemxi} for $u\ge 23$.
This demonstrates the result for $u\ge 23$. For $1 \le u\le 23$ we confirm the result by graphing the
four quantities in question multiplied by $u$. 
\end{proof}

\begin{proof}[Proof of Theorem \ref{thm1}]
For $u\ge 5$ the result follows from Theorem \ref{thm2} and Lemma \ref{lemalpha}.
For $1\le u \le 5$, we calculate $\rho(u)$ as described in Section \ref{SecNum}
and graph the three expressions as in Figure \ref{fig1}.
\end{proof}

\section{Numerical calculation of $\rho(u)$}\label{SecNum}

In the range $1\le u \le 50$, we calculate $\rho(u)$ with the algorithm of Marsaglia, Zaman and Marsaglia \cite{MZM}: 
To approximate $\rho(u)$ in the interval $[n,n+1]$, where $n$ is an integer, 
we find the coefficients of the Taylor polynomial of order $250$ centered
at $n+1/2$, recursively from the coefficients of the polynomial centered at $n-1/2$.

For larger $u$, we evaluate the inverse Laplace integral numerically. After the change of variables $t=\tau \sqrt{I''(\xi)}$,
equation \eqref{eqlapinv} implies
$$
\frac{\rho(u)}{\tilde{\rho}(u)} 
= \int_{-\infty}^\infty \frac{1}{\sqrt{2 \pi}}   \exp\left\{I(\xi-i t/\sqrt{I''(\xi)})-I(\xi)+ui t/\sqrt{I''(\xi)}\right\} dt.
$$
The proof of Theorem \ref{thm2} shows that the real part of the integrand approaches the standard normal density function as $u$ grows.
Numerical integration in Mathematica converges well for $u\ge 15$.

With this way of calculating $\rho(u)$, our values match those in \cite{LunWat} to the accuracy they claim, 
namely at least five significant digits.

\section*{Acknowledgment}
The author thanks G\'erald Tenenbaum and the anonymous referee for helpful suggestions.

\end{document}